\documentclass{svproc}
\usepackage[normalem]{ulem}
\usepackage[T1]{fontenc}
\usepackage{amsfonts}
\usepackage{tikz}
\usepackage{textcomp}
\usepackage{graphics}
\usepackage{amssymb}
\usepackage{amsmath}

\usepackage{amsthm}
\usepackage{multirow}
\newtheorem*{theorem*}{Theorem}
\newtheorem*{lemma*}{Lemma}
\newtheorem{thm}{Theorem}
\newtheorem{lem}[thm]{Lemma}
\theoremstyle{remark}
\newtheorem{rem}[thm]{Remark}
\usepackage{booktabs}

\let\counterwithin\relax
\usepackage{chngcntr}
\counterwithin{table}{section}
\usepackage{array}
\newcommand\D[1]{\mathrm{d}#1\,}
\newcommand{\R}{\mathbb R}

\newcommand{\N}{\mathbb N}

\newcommand{\cL}{\mathcal L}
\newcommand{\Lis}{\cL\mathrm{is}}
\DeclareMathOperator{\blockdiag}{blockdiag}
\DeclareMathOperator{\spann}{span}
\DeclareMathOperator{\diam}{diam}
\DeclareMathOperator{\supp}{supp}
\usepackage{xcolor}
\newcommand{\eps}{\epsilon}
\newcommand{\kron}{\otimes}
\newcommand{\new}[1]{{#1}}

\newcommand{\identity}{\mathrm{Id}}
\newcommand{\grad}{\nabla}
\newcommand{\logt}{\log^2\!}
\renewcommand{\vec}{\mathbf}
\let\Vec\undefined
\DeclareMathOperator*{\Vec}{Vec}
\newcommand*{\thead}[1]{\multicolumn{1}{r}{#1}}
\def\Vhrulefill{\leavevmode\leaders\hrule height 0.7ex depth \dimexpr0.4pt-0.7ex\hfill\kern0pt}

\newcommand{\be}{\begin{equation}}
\newcommand{\ee}{\end{equation}}

\usepackage{url}

\begin{document}
\mainmatter              
\title{A \new{parallel} algorithm for solving linear parabolic evolution equations}
\titlerunning{A \new{parallel} algorithm for parabolic evolution equations}  
%
\author{Raymond van Veneti\"e \and Jan Westerdiep}
\authorrunning{R.~van Veneti\"e, J.~Westerdiep} 
%
\tocauthor{Raymond van Veneti\"e and Jan Westerdiep}
\institute{Korteweg--de Vries (KdV) Institute for Mathematics, University of Amsterdam,\\
PO Box 94248, 1090 GE Amsterdam, The Netherlands \\
\email{r.vanvenetie@uva.nl}, ~~ \email{j.h.westerdiep@uva.nl}}

\maketitle              

\begin{abstract}
We present an algorithm for the solution of a simultaneous space-time
discretization of linear parabolic evolution equations with a symmetric
differential operator in space. Building on earlier work, we recast this
discretization into a Schur-complement equation whose solution
is a quasi-optimal approximation to the weak solution of the equation at hand.
Choosing a tensor-product discretization, we arrive at a remarkably simple linear system.
Using wavelets in time and standard finite elements in space,
we solve the resulting system in linear complexity on a single processor,
and in \new{poly}logarithmic complexity when parallelized in both space and time.
We complement these theoretical findings with large-scale parallel
computations showing the effectiveness of the method.


~\\
\new{\textbf{Supplementary material:} Source code is available at~\cite{VanVenetie2020b}.}
\end{abstract}

\section{Introduction}
This paper deals with solving parabolic evolution equations in a time-parallel fashion using tensor-product discretizations.
Time-parallel algorithms for solving parabolic evolution equations have \new{become a focal point} following the enormous increase in parallel computing
power.
Spatial parallelism is a ubiquitous component in large-scale computations, but
when spatial parallelism is exhausted, parallelization of the time axis is of interest.

Time-stepping methods first discretize the problem in space, and then solve the arising system
of coupled ODEs sequentially\new{, immediately revealing a primary source of difficulty for time-parallel computation.}

Alternatively, one can solve simultaneously in space \emph{and} time. Originally introduced in \cite{Babuska1989a,Babuska1990a},
these space-time methods are very flexible: some can guarantee quasi-best approximations, \new{meaning that their error is proportional to that of the best approximation from the trial space}
\cite{Andreev2013,Devaud2018,Fuhrer2019,Steinbach2020a}, or drive adaptive routines \new{\cite{Steinbach2018,Rekatsinas2019}}.
Many are
especially well-suited for time-parallel computation \cite{Gander2014,Neumuller2019}. Since the first significant contribution
to time-parallel algorithms~\cite{Nievergelt1964} in 1964, many methods suitable for parallel
computation have surfaced; see the review~\cite{Gander2015}.

\subsubsection*{Parallel complexity.}
The (serial) complexity of an algorithm measures asymptotic runtime on a single processor
in terms of the input size. \emph{Parallel complexity} measures asymptotic runtime
given \emph{sufficiently many} parallel processors having access to a
shared memory, i.e., assuming there are no communication costs.

In the \new{current} context of tensor-product discretizations \new{of parabolic PDEs}, \new{we denote} with $N_t$
and $N_{\vec x}$ the \new{number of unknowns in time and space respectively}.

The parareal method~\cite{Lions2001} aims at time-parallelism by alternating a
serial coarse-grid solve with fine-grid computations in parallel. This
way, \new{each iteration has} a time-parallel complexity of
$\mathcal O( \sqrt {N_t} N_{\vec x})$, \new{and combined}
with parallel multigrid in space, a parallel complexity of
$\mathcal O(\sqrt{N_t} \log N_{\vec x})$.
The popular MGRIT algorithm extends these ideas to multiple levels in time; cf.~\cite{Falgout2014}.

Recently, Neum\"uller and Smears proposed an \new{iterative} algorithm \new{that uses a Fast Fourier Transform in time. Each iteration runs serially in $\mathcal O(N_t \log(N_t) N_{\vec x})$ and parallel in time, in $\mathcal O(\log(N_t) N_{\vec x})$.}
By \new{also} incorporating parallel multigrid in space, its parallel
runtime may be reduced to $\mathcal O(\log N_t + \log N_{\vec x})$.

\subsubsection*{Our contribution.}
In this paper, we study a variational formulation introduced in~\cite{Stevenson2020} which was
based on work by Andreev~\cite{Andreev2013,Andreev2016}.
\new{Recently in \cite{Stevenson2020b,VanVenetie2020}, we studied this formulation in the context of space-time
adaptivity and its efficient implementation in serial and on shared-memory parallel computers.}
The current paper instead focuses on its massively parallel implementation and time-parallel performance.

Our method has remarkable similarities with the approach of~\cite{Neumuller2019},
\new{and the most essential difference is the substitution of their Fast Fourier Transform by our Fast Wavelet Transform.}
The strengths of both methods include a solid inf-sup theory that \new{enables}
quasi-optimal approximate solutions from the trial space, ease of implementation,
and excellent parallel performance in practice.

Our method has another strength: based on a wavelet transform, \new{for fixed algebraic tolerance} it runs
serially in linear complexity.
Parallel in time, it runs in complexity $\mathcal O(\log(N_t) N_{\vec x})$;
parallel in \emph{space and time}, \new{in}
$\mathcal O(\log (N_t N_{\vec x}))$. \new{Moreover, when solving to an algebraic error proportional to the discretization error, incorporating a \emph{nested iteration} (cf.~\cite[Ch.~5]{Hackbusch1985}) results in complexities $\mathcal O(N_t N_{\vec x})$, $\mathcal O(\log (N_t) N_{\vec x})$, and $\mathcal O(\logt (N_t N_{\vec x}))$ respectively. This is on par with best-known results on parallel complexity for elliptic problems; see also~\cite{Brandt1981}.}

\subsubsection*{Organization of this paper.}
In \S\ref{sec:problem}, we formally introduce the problem, derive a saddle-point formulation,
and provide sufficient conditions for quasi-optimality of discrete solutions.
In \S\ref{sec:solving}, we detail on the efficient computation of these discrete solutions.
In \S\ref{sec:concrete} we take a concrete example---the reaction-diffusion equation---and
analyze the serial and parallel complexity of our algorithm.
In \S\ref{sec:numres}, we test these theoretical findings in practice.
We conclude in \S\ref{sec:conclusion}.

\subsubsection*{Notations.}
For normed linear spaces $U$ and $V$, in this paper for convenience over $\R$, $\cL(U,V)$ will denote the space of bounded linear mappings $U \rightarrow V$ endowed with the operator norm $\|\cdot\|_{\cL(U,V)}$. The subset of invertible operators in $\cL(U,V)$  with inverses in $\cL(V,U)$ will be denoted as $\Lis(U,V)$.

Given a finite-dimensional subspace $U^\delta$ of a normed linear space $U$, we denote the trivial embedding $U^\delta \to U$ by $E_U^\delta$.
For a basis $\Phi^\delta$---viewed formally as a column vector---of $U^\delta$, we define the \emph{synthesis operator} as
\[
  \mathcal F_{\Phi^\delta}: \R^{\dim U^\delta} \to U^\delta: {\bf c} \mapsto {\bf c}^\top \Phi^\delta =: \sum_{\phi \in \new{\Phi^\delta}} c_\phi \phi.
\vspace{-1em}
\]
Equip $\R^{\dim U^\delta}$ with the Euclidean inner product and identify $(\R^{\dim U^\delta})'$ with $\R^{\dim U^\delta}$ using the corresponding Riesz map. We find the adjoint of $\mathcal F_{\Phi^\delta}$, the \emph{analysis operator}, to satisfy
\[
({\mathcal F_{\Phi^\delta}})': (U^\delta)' \to \R^{\dim U^\delta} : f \mapsto f(\Phi^\delta) := [f(\phi)]_{\phi \in \new{\Phi^\delta}}.
\]

For quantities $f$ and $g$, by $f \lesssim g$, we mean that $f \leq C \cdot g$ with a constant that does not depend on parameters that $f$ and $g$ may depend on. By $f \eqsim g$, we mean that $f \lesssim g$ and $g \lesssim f$.
For matrices ${\bf A}$ and ${\bf B} \in \R^{N \times N}$, by ${\bf A} \eqsim {\bf B}$ we will denote \emph{spectral equivalence}, i.e.~${\bf x}^\top {\bf A} {\bf x} \eqsim {\bf x}^\top {\bf B} {\bf x}$ for all ${\bf x} \in \R^N$.

\section{Quasi-optimal approximations to the parabolic problem}\label{sec:problem}
Let $V,H$ be separable Hilbert spaces of functions on some spatial domain
such that $V$ is continuously embedded in $H$, i.e.~$V \hookrightarrow H$, with dense compact embedding.
Identifying $H$ with its dual yields the Gelfand triple
$V \hookrightarrow H \simeq H' \hookrightarrow V'$.

For a.e.
\[ t \in I:=(0,T), \]
let $a(t;\cdot,\cdot)$ denote a bilinear form on $V \times V$ so that for
any $\eta,\zeta \in V$, $t \mapsto a(t;\eta,\zeta)$ is measurable on $I$,
and such that for a.e.~$t\in I$,
\begin{alignat*}{3}
|a(t;\eta,\zeta)|
& \lesssim  \|\eta\|_{V} \|\zeta\|_{V} \quad &&(\eta,\zeta \in V) \quad &&\text{({\em boundedness})},
\nonumber \\
 a(t;\eta,\eta) &\gtrsim \|\eta\|_{V}^2 \quad
&&(\eta \in {V}) \quad &&\text{({\em coercivity})}.
\end{alignat*}

With $(A(t) \cdot)(\cdot) := a(t; \cdot, \cdot) \in \Lis({V},V')$, given a forcing function $g$ and initial value $u_0$, we want to solve the {\em parabolic initial value problem} of
\begin{equation} \label{11}
\text{finding $u: I \to V$ such that} \quad
\left\{
\begin{array}{rl}
\frac{\D u}{\D t}(t) +A(t) u(t)& = g(t) \quad(t \in I),\\
u(0) & = u_0.
\end{array}
\right.
\end{equation}

\subsection{An equivalent self-adjoint saddle-point system}
In a simultaneous space-time variational formulation, the parabolic problem reads as finding $u$ from a suitable space of functions of time and space s.t.
\[
(Bw)(v):=\int_I
\langle { \textstyle \frac{\D w}{\D t}}(t), v(t)\rangle_H +
a(t;w(t),v(t)) \D t = \int_I
\langle g(t), v(t)\rangle_H =:g(v)
\]
for all $v$ from another suitable space of functions of time and space.
One possibility to enforce the initial condition is by testing against additional test functions.
\begin{thm}[{\cite{Schwab2009}}] \label{thm0} With $X:=L_2(I;{V}) \cap H^1(I;V')$, $Y:=L_2(I;{V})$, we have
$$
\left[\begin{array}{@{}c@{}} B \\ \gamma_0\end{array} \right]\in \Lis(X,Y' \times H),
$$
where for $t \in \bar{I}$, $\gamma_t\colon u \mapsto u(t,\cdot)$ denotes the trace map.
In other words,
\be \label{x12}
\text{finding $u \in X$ s.t.} \quad (Bu, \gamma_0 u) = (g, u_0) \quad \text{given} \quad (g, u_0) \in Y' \times H
\ee
is a well-posed simultaneous space-time variational formulation of \eqref{11}.
\end{thm}
We define $A \in \Lis(Y, Y')$ and $\partial_t \in \Lis(X, Y')$ as
\[
    (Au)(v) := \int_I a(t; u(t), v(t)) \D t, \quad \text{and} \quad \partial_t := B - A.
\]
Following~\cite{Stevenson2020}, we assume that $A$ is \emph{symmetric}.
We can reformulate~\eqref{x12} as the self-adjoint saddle point problem
\begin{equation} \label{m0}
\text{finding $(v, \sigma, u) \in Y \times H \times X$ s.t.} \quad
\left[\begin{array}{@{}ccc@{}} A & 0 & B\\ 0 & \identity & \gamma_0\\ B' & \gamma_0' & 0\end{array}\right]
\left[\begin{array}{@{}c@{}} v \\ \sigma \\ u \end{array}\right]=
\left[\begin{array}{@{}c@{}} g \\ u_0 \\ 0 \end{array}\right].
\end{equation}
By taking a Schur complement w.r.t.~the $H$-block, we can reformulate this as
\begin{equation}\label{m1}
\text{finding $(v, u) \in Y \times X$ s.t.} \quad
\left[\begin{array}{@{}cc@{}} A & B\\ B' & -\gamma_0' \gamma_0\end{array}\right]
\left[\begin{array}{@{}c@{}} v  \\ u \end{array}\right]=
\left[\begin{array}{@{}c@{}} g \\ -\gamma_0' u_0 \end{array}\right].
\end{equation}

We equip $Y$ and $X$ with \emph{`energy'-norms}
$$
\|\cdot\|_Y^2:=(A \cdot)(\cdot),\quad
\|\cdot\|_X^2:=\|\partial_t \cdot\|_{Y'}^2 + \|\cdot\|_Y^2 + \|\gamma_T \cdot\|_H^2,
$$
which are equivalent to the canonical norms on $Y$ and $X$.

\subsection{Uniformly quasi-optimal Galerkin discretizations}

Our numerical approximations will be based on the saddle-point formulation~\eqref{m1}.
Let $(Y^\delta, X^\delta)_{\delta \in \Delta}$ be a collection of closed subspaces of $Y \times X$ satisfying
\be \label{infsup1}
    X^\delta \subset Y^\delta, \quad \partial_t X^\delta \subset Y^\delta \quad (\delta \in \Delta),
\ee
and
\be \label{infsup2}
    1 \geq \gamma_\Delta := \inf_{\delta \in \Delta} \inf_{0 \not= u \in X^\delta} \sup_{0 \not= v \in Y^\delta} \frac{(\partial_t u)(v)}{\|\partial_t u\|_{Y'} \|v\|_Y} > 0.
\ee
\begin{rem}
In \cite[\S 4]{Stevenson2020}, these conditions were verified for $X^\delta$ and $Y^\delta$ being
tensor-products of \new{(locally refined) finite element spaces in time and space}.
In~\cite{Stevenson2020b}, we relax these conditions to $X^\delta_t$ and $Y^\delta$ being
\emph{adaptive sparse grids}, allowing adaptive refinement locally in space \emph{and} time
simultaneously.
\end{rem}

For $\delta \in \Delta$, let
$(v^\delta, \overline u^\delta) \in Y^\delta \times X^\delta$ solve the Galerkin discretization of~\eqref{m1}:
\be \label{m8}
\left[\begin{array}{@{}ccc@{}}{E_Y^\delta}' A E_Y^\delta& {E_Y^\delta}' B E^\delta_X\\ {E^\delta_X}' B' E_Y^\delta& -{E^\delta_X}' \gamma_0' \gamma_0 E^\delta_X \end{array}\right]
\left[\begin{array}{@{}c@{}} v^\delta \\ \overline u^\delta \end{array}\right]=
\left[\begin{array}{@{}c@{}} {E^\delta_Y}' g \\ -{E^\delta_X}' \gamma_0' u_0 \end{array}\right].
\ee
The solution $(v^\delta, \overline u^\delta)$ of \eqref{m8} exists uniquely, and exhibits \emph{uniform quasi-optimality} in that $\|u - \overline u^\delta\|_X \leq \gamma_\Delta^{-1} \inf_{u_\delta \in X^\delta} \|u - u_\delta\|_X$ \new{for all $\delta \in \Delta$}.

Instead of solving a matrix representation of \eqref{m8} using e.g.~preconditioned MINRES,
we will opt for a computationally more attractive method. By taking the Schur complement
w.r.t.~the $Y^\delta$-block in~\eqref{m8}, and replacing $({E_Y^\delta}' A E_Y^\delta)^{-1}$
in the resulting formulation by a \emph{preconditioner} $K_Y^\delta$ that can be applied
cheaply, we arrive at the \emph{Schur complement formulation} of finding $u^\delta \in X^\delta$ s.t.
\be \label{m9}
   \underbrace{{E^\delta_X}'(B' E^\delta_Y K_Y^\delta {E^\delta_Y}' B+\gamma_0'\gamma_0)E^\delta_X}_{=: S^\delta} u^\delta = \underbrace{{E^\delta_X}' (B' E^\delta_Y K_Y^\delta {E^\delta_Y}' g+\gamma_0' u_0)}_{=: f^\delta}.
\ee
The resulting operator $S^\delta \in \Lis(X^\delta, {X^\delta}')$ is self-adjoint and elliptic.
Given a self-adjoint operator $K_Y^\delta \in \cL({Y^\delta}',Y^\delta)$ satisfying, for some $\kappa_\Delta \geq 1$,
\be \label{eqn:spectral}
    \frac{\big(({K_Y^\delta})^{-1} v\big)(v)}{(A v\big)(v)} \in [\kappa_\Delta^{-1}, \kappa_\Delta] \quad (\delta \in \Delta, ~~ v \in Y^\delta),
\ee
the solution $u^\delta$ of \eqref{m9} exists uniquely as well. In fact, the following holds.
\begin{thm}[{\cite[Rem.~3.8]{Stevenson2020}}]\label{thm:quasiopt}
Take $(Y^\delta \times X^\delta)_{\delta \in \Delta}$ satisfying~\eqref{infsup1}--\eqref{infsup2}, and
$K_Y^\delta$ satisfying~\eqref{eqn:spectral}. Solutions $u^\delta \in X^\delta$ of~\eqref{m9} \new{are uniformly quasi-optimal, i.e.}
\[
    \|u - u^\delta\|_X \leq \frac{\kappa_\Delta}{\gamma_{\Delta}} \inf_{u_{\delta} \in X^\delta} \|u - u_\delta\|_X \quad \new{(\delta \in \Delta)}.
\]
\end{thm}

\section{Solving efficiently on tensor-product discretizations}\label{sec:solving}
From now on, we assume that $X^\delta := X^\delta_t \kron X^\delta_{\vec x}$ and $Y^\delta := Y^\delta_t \kron Y^\delta_{\vec x}$ are \emph{tensor-products}, and for ease of presentation, we assume that the spatial discretizations on $X^\delta$ and $Y^\delta$ coincide, i.e.~$X^\delta_{\vec x} = Y^\delta_{\vec x}$, reducing~\eqref{infsup1} to $X^\delta_t \subset Y^\delta_t$ and $\tfrac{\D}{\D t} X^\delta_t \subset Y^\delta_t$.

We equip
$X^\delta_t$ with a basis $\Phi^\delta_t$, $X^\delta_{\vec x}$ with $\Phi_{\vec x}^\delta$, and $Y_t^\delta$ with $\Xi^\delta$.

\subsection{Construction of $K_Y^\delta$}
\label{sec:DY-precond}
Define ${\bf O} := \langle \Xi^\new{\delta}, \Xi^\new{\delta} \rangle_{L_2(I)}$ and ${\bf A}_{\vec x} := \langle \Phi_{\vec x}^\delta, \Phi_{\vec x}^\delta \rangle_V$.
Given ${\bf K}_{\vec x} \eqsim {\bf A}_{\vec x}^{-1}$
uniformly in $\delta \in \Delta$, define
\[
  {\bf K}_Y := {\bf O}^{-1} \kron {\bf K}_{\vec x}.
\]
Then, the preconditioner $K_Y^\delta := \mathcal F_{\Xi^\delta \kron \Phi_{\vec x}^\delta} {\bf K}_Y (\mathcal F_{\Xi^\delta \kron \Phi_{\vec x}^\delta})' \in \cL({Y^\delta}', Y^\delta)$ satisfies~\eqref{eqn:spectral}\new{; cf.~\cite[\S5.6.1]{Stevenson2020b}}.

When $\Xi^\delta$ is orthogonal, ${\bf O}$ is diagonal and can be inverted exactly. For standard finite element bases $\Phi^\delta_{\vec x}$,  suitable $\mathbf K_{\vec x}$ that can be applied efficiently (at cost linear in the discretization size) are provided by symmetric multigrid methods.

\subsection{Preconditioning the Schur complement formulation} \label{sec:precond}
We will solve a matrix representation of~\eqref{m9} with an iterative solver, thus requiring a preconditioner.
Inspired by the constructions of \cite{Andreev2016,Neumuller2019}, we build an \emph{optimal} self-adjoint coercive preconditioner $K^\delta_X \in \cL( {X^\delta}', X^\delta)$ as a wavelet-in-time block-diagonal matrix with multigrid-in-space blocks.

Let $U$ be a separable Hilbert space of functions over some domain.
A given collection $\Psi = \{\psi_\lambda\}_{\lambda \in \vee_\Psi}$ is a \emph{Riesz basis} for $U$ when
\[
    \overline{\spann \Psi} = U, \quad \text{and} \quad \| {\bf c}\|_{\ell_2(\vee_\Psi)} \eqsim \|{\bf c}^\top \Psi\|_U \quad \text{for all} \quad {\bf c} \in \ell_2(\vee_\Psi).
\]
Thinking of $\Psi$ being a basis of wavelet-type, for indices $\lambda \in \vee_\Psi$, its \emph{level} is denoted $|\lambda| \in \N_0$. We call $\Psi$ \emph{uniformly local} when for all $\lambda \in \vee_\Psi$,
\[
    \diam (\supp \psi_\lambda) \lesssim 2^{-|\lambda|} ~~\text{and}~~  \# \{\mu \in \vee_\Psi : |\mu| = |\lambda|, |\supp \psi_\mu \cap  \supp \psi_\lambda | > 0\} \lesssim 1.
\]

Assume $\Sigma := \{ \sigma_\lambda : \lambda \in \vee_\Sigma\}$ is a uniformly local Riesz basis for $L_2(I)$ with $\{2^{-|\lambda|}\sigma_\lambda : \lambda \in \vee_\Sigma\}$ Riesz for $H^1(I)$.
Writing $w \in X$ as $\sum_{\lambda \in \vee_\Sigma} \sigma_\lambda \kron w_\lambda$ for some $w_\lambda \in V$,
we define the bounded, symmetric, and coercive bilinear form
\[
  (D_X \sum_{\lambda \in \vee_\Sigma} \sigma_\lambda \kron w_\lambda)(\sum_{\mu \in \vee_\Sigma} \sigma_\mu \kron v_\mu) := \sum_{\lambda \in \vee_\Sigma} \langle w_\lambda, v_\lambda \rangle_V + 4^{|\lambda|} \langle w_\lambda, v_\lambda \rangle_{V'}.
\]
The operator $D_X^\delta := {E_X^\delta}' D_X E_X^\delta$ is in $\Lis(X^\delta, {X^\delta}')$. Its
norm and that of its inverse are bounded uniformly in $\delta \in \Delta$.
When $X^\delta = \spann \Sigma^\delta \kron \Phi_{\vec x}^\delta$ for some $\Sigma^\delta \new{:= \{ \sigma_\lambda \colon \lambda \in \vee_{\Sigma^\delta}\}}\subset \Sigma$,
the matrix representation of $D_X^\delta$ \new{w.r.t.~$\Sigma^\delta \kron \Phi_{\vec x}^\delta$ is}
\[
  \new{({\mathcal F_{\Sigma^\delta \kron \Phi^\delta}})}' D_X^\delta \mathcal F_{\Sigma^\delta \kron \Phi^\delta} =: \mathbf D_X^\delta = \blockdiag[{\bf A}_{\vec x} + 4^{|\lambda|} \langle \Phi_{\vec x}^\delta, \Phi_{\vec x}^\delta \rangle_{V'}]_{\lambda \in \vee_{\Sigma^\delta}}.
\]
\begin{thm}[{\cite[\new{\S5.6.2}]{Stevenson2020b}}]
  \label{thm:DX-precond}
      Define ${\bf M}_{\vec x} := \langle \Phi_{\vec x}^\delta, \Phi_{\vec x}^\delta \rangle_H$.
  When we have matrices ${\bf K}_j \eqsim (\mathbf A_{\vec x} + 2^{j} \mathbf M_{\vec x})^{-1}$ uniformly in $\delta \in \Delta$ and $j \in \N_0$, it follows that
  \[
    \mathbf D_X^{-1} \eqsim {\bf K}_X := \blockdiag[{\bf K}_{|\lambda|} \mathbf A_{\vec x} {\bf K}_{|\lambda|}]_{\lambda \in \vee_{\Sigma^\delta}}.
  \]
  This yields an optimal preconditioner $K_X^\delta := {\mathcal F_{\Sigma^\delta \kron \Phi^\delta}} {\bf K}_X ({\mathcal F_{\Sigma^\delta \kron \Phi^\delta}})' \in \Lis({X^\delta}', X^\delta)$.
\end{thm}
In~\cite{Olshanskii2000} it was shown that \new{under a `full-regularity' assumption, for quasi-uniform meshes}, a multiplicative multigrid method yields \new{${\bf K}_j$} satisfying the conditions of Thm.~\ref{thm:DX-precond}, which can moreover be applied in linear time.

\subsection{Wavelets in time}\label{sec:wavelets}
The preconditioner ${\bf K}_X$ requires $X^\delta_t$ to be equipped with a \emph{wavelet} basis
$\Sigma^\delta$, whereas one typically uses a different \new{(single-scale)} basis $\Phi^\delta_t$ on
$X^\delta_t$. To bridge this gap, a basis transformation from $\Sigma^\delta$ to $\Phi^\delta_t$
is required. We define the wavelet transform as ${\bf W}_t := ({\mathcal F_{\Phi_t^\delta}})^{-1} {\mathcal F_{\Sigma^\delta}}$.\footnote{In literature, this transform is typically called an \emph{inverse wavelet transform}.}

Define $V_j := \spann \{\sigma_\lambda \in \Sigma: |\lambda| \leq j\}$. Equip each $V_j$ with a (single-scale) basis $\Phi_j$, and assume that $\new{\Phi^\delta_t} := \Phi_J$ for some $J$, so that $X^\delta_t := V_J$. Since $V_{j+1} = V_j \oplus \spann \Sigma_j$ where $\Sigma_j := \{\sigma_\lambda : |\lambda| = j\}$, there exist matrices ${\bf P}_j$ and ${\bf Q}_j$ such that $\Phi_j^\top = \Phi_{j+1}^\top {\bf P}_j$ and $\Psi_j^\top = \Phi_{j+1}^\top {\bf Q}_j$, with ${\bf M}_j := [{\bf P}_j | {\bf Q}_j]$ invertible.

Writing $v \in V_J$ in both forms
$v = {\bf c}_0^\top \Phi_0 + \sum_{j=0}^{J-1} {\bf d}_j^\top \Psi_j$ and $v = {\bf c}_J^\top \Phi_J$,
the basis transformation ${\bf W}_t := {\bf W}_J$ mapping wavelet coordinates $({\bf c}_0^\top, {\bf d}_0^\top, \ldots, {\bf d}_{J-1}^\top)$ to single-scale coordinates ${\bf c}_J$ satisfies
\be \label{eqn:waveletcomposition}
{\bf W}_J = {\bf M}_{J-1} \begin{bmatrix}
 {\bf W}_{J-1} & {\bf 0} \\ {\bf 0} & {\mathrm{\bf Id}}
\end{bmatrix}
, \quad \text{and} \quad {\bf W}_0 := {\mathrm {\bf Id}}.
\ee
Uniform locality of $\Sigma$ implies \emph{uniform sparsity} of the ${\bf M}_j$, i.e.~with $\mathcal O(1)$ nonzeros per row and column.
Then, assuming a geometrical increase in $\dim V_j$ in terms of $j$, which is true in the concrete setting below,
matrix-vector products ${\bf x} \mapsto {\bf W}_t {\bf x}$ can be performed (serially) in linear complexity; \new{cf.~\cite{Stevenson2003}}.

\subsection{Solving the system}\label{sub:solving}
The matrix representation of $S^\delta$ and $f^\delta$ from~\eqref{m9} w.r.t.~a basis  $\Phi_t^\delta \kron \Phi_{\vec x}^\delta$ of $X^\delta$ is
\[
{\bf S} := \new{({\mathcal F_{\Phi^\delta_t \kron \Phi^\delta_{\vec x}}})}' S^\delta \new{\mathcal F_{\Phi_t^\delta \kron \Phi_{\vec x}^\delta}} \quad \text{and} \quad
{\bf f} := \new{({\mathcal F_{\Phi_t^\delta \kron \Phi_{\vec x}^\delta}})}' f^\delta.
\]
\new{Envisioning an iterative solver, using \S\ref{sec:precond} we have a preconditioner in terms of the wavelet-in-time basis $\Sigma^\delta \kron \Phi_{\vec x}^\delta$, with which} their
matrix representation is
\be \label{eqn:schur-mat}
  \new{\hat {\bf S}} := \new{({\mathcal F_{\Sigma^\delta \kron \Phi_{\vec x}^\delta}})}' S^\delta \mathcal \new{\mathcal F_{\Sigma^\delta \kron \Phi_{\vec x}^\delta}} \quad \text{and} \quad
  \new{\hat {\bf f}} := \new{({\mathcal F_{\Sigma^\delta \kron \Phi_{\vec x}^\delta}})}' f^\delta.
\ee
These two forms are related: with the wavelet transform ${\bf W} := {\bf W}_t \kron {\mathrm{\bf Id}}_{\vec x}$, we have $\hat {\bf S} = {\bf W}^\top {\bf S} {\bf W}$ and $\hat {\bf f} = {\bf W}^\top {\bf f}$, and the matrix representation of~\eqref{m9} becomes
\begin{equation}
\label{eqn:wit-mat}
 \text{finding ${\bf w}$} \quad \text{s.t.} \quad \hat{\bf S}{\bf w} = \hat {\bf f}.
\end{equation}
We can then recover the solution in single-scale coordinates as ${\bf u} = {\bf W} {\bf w}$.

We use Preconditioned Conjugate Gradients (PCG), with preconditioner ${\bf K}_X$, to solve~\eqref{eqn:wit-mat}. Given an algebraic error tolerance $\eps > 0$ and current guess ${\bf w}_k$,
we monitor
${\bf r}_k^\top{\bf K}_X {\bf r}_k \leq \eps^2$ where ${\bf r}_k := \hat {\bf f} - \hat {\bf S} {\bf w}_k$.
This data is available within PCG, and constitutes a stopping criterium:
with $u_k^\delta := \mathcal F_{\Sigma^\delta \kron \Phi_{\vec x}^\delta} {\bf w}_k \in X^\delta$, we see
\be \label{eqn:stopping}
    {\bf r}_k^\top {\bf K}_X {\bf r}_k
    = (f^\delta - S^\delta u_k^\delta)(K_X^\delta (f^\delta - S^\delta u_k^\delta)) \eqsim \|u^\delta - u_k^\delta\|_{X}^2
\ee
\new{with $\eqsim$ following from \cite[(4.12)]{Stevenson2020b},}
so that the algebraic error satisfies $\|u^\delta - u^\delta_k\|_X \lesssim \eps$.

\section{A concrete setting: the reaction-diffusion equation}\label{sec:concrete}
On a bounded Lipschitz domain $\Omega \subset \R^d$, take $H := L_2(\Omega)$, $V := H^1_0(\Omega)$, and
\[
    a(t; \eta, \zeta) := \int_\Omega {\bf D} \grad \eta \cdot \grad \zeta + c \eta \zeta \D{\vec x}
\]
where ${\bf D} = {\bf D}^\top \in \R^{d \times d}$ is positive definite, and $c \geq 0$.\footnote{This is easily generalized to variable coefficients, but notation becomes more obtuse.} We note that $A(t)$ is symmetric and coercive. W.l.o.g.~we take $I := (0,1)$, i.e.~$T:=1$.

Fix $p_t, p_{\vec x} \in \N$.
With $\{\mathcal T_I\}$ the family of \new{quasi-uniform} partitions of $I$ into subintervals, and $\{\mathcal T_\Omega\}$ that of conforming quasi-uniform triangulations of $\Omega$, we define $\Delta$ as the collection of pairs $(\mathcal T_I, \mathcal I_\Omega)$. We construct our trial- and test spaces as
\[
  X^\delta := X^\delta_t \kron X^\delta_{\vec x}, \quad Y^\delta := Y^\delta_t \kron X^\delta_{\vec x},
\]
where, with $\mathbb P_{p}^{-1}(\mathcal T)$ denoting the space of piecewise degree-$p$ polynomials on $\mathcal T$,
\[
  X^\delta_t := H^1(I) \cap \mathbb P_{p_t}^{-1}(\mathcal T_I), \quad
  X^\delta_{\vec x} := H^1_0(\Omega) \cap \mathbb P_{p_{\vec x}}^{-1}(\mathcal T_\Omega), \quad
  Y^\delta_t := \mathbb P_{p_t}^{-1}(\mathcal T_I).
\]
These spaces satisfy condition \eqref{infsup1}, with coinciding spatial discretizations on $X^\delta$ and $Y^\delta$. For this choice of $\Delta$, inf-sup condition \eqref{infsup2} follows from~\cite[Thm.~4.3]{Stevenson2020}.

For $X^\delta_t$, we choose $\Phi_t^\delta$ to be the Lagrange basis of degree $p_t$ on $\mathcal T_I$; for $X^\delta_{\vec x}$, we choose $\Phi_{\vec x}^\delta$ to be \new{that} of degree $p_{\vec x}$ on $\mathcal T_\Omega$.
An orthogonal basis $\Xi^\delta$ for $Y^\delta_t$ may be built as piecewise shifted Legendre polynomials of degree $p_t$ w.r.t.~$\mathcal T_I$.

For $p_t = 1$, one finds a suitable wavelet basis $\Sigma$ in~\cite{Stevenson1998}.
For $p_t > 1$, one can either split the system into lowest- and higher-order parts and
perform the transform on the lowest-order part only, or construct
higher-order wavelets directly; cf.~\cite{Dijkema2009}.

Owing to the tensor-product structure of $X^\delta$ and $Y^\delta$ and of the operators $A$ and $\partial_t$, the matrix representation of our formulation becomes remarkably simple.

\begin{lem} \label{lem:schurmat}
Define ${\bf g} := \new{({\mathcal F_{\Xi^\delta \kron \Phi_{\vec x}^\delta}})}' g$, ${\bf u}_0 := \Phi_t^\delta(0) \kron \langle u_0, \Phi_{\vec x}^\delta \rangle_{L_2(\Omega)}$, and
\begin{alignat*}{2}
  {\bf T} &:= \langle \tfrac{\D}{\D t} \Phi_t^\delta, \Xi^\delta \rangle_{L_2(I)},& \quad {\bf N} &:= \langle \Phi_t^\delta, \Xi^\delta \rangle_{L_2(I)}, \\
  {\bf \Gamma}_0 &:= \Phi_t^\delta(0) [\Phi_t^\delta(0)]^\top,& \quad 
  {\bf M}_{\vec x} &:= \langle \Phi_{\vec x}^\delta, \Phi_{\vec x}^\delta \rangle_{L_2(\Omega)}, \\ 
  {\bf A}_{\vec x} &:=  \langle  {\bf D}\nabla \Phi_{\vec x}^\delta, \nabla \Phi_{\vec x}^\delta \rangle_{L_2(\Omega)} + c {\bf M}_{\vec x},& \quad
  {\bf B} &:= {\bf T} \kron {\bf M}_{\vec x} + {\bf N} \kron {\bf A}_{\vec x}.
\end{alignat*}
\new{With} ${\bf K}_Y := {\bf O}^{-1} \kron {\bf K}_{\vec x}$ from \S\ref{sec:DY-precond},  we can write ${\bf S}$ and ${\bf f}$ \new{from \S\ref{sub:solving}} as
\[
    {\bf S} = {\bf B}^\top {\bf K}_Y {\bf B} + {\bf \Gamma}_0 \kron {\bf M}_{\vec x}, \quad
    {\bf f} = {\bf B}^\top {\bf K}_Y {\bf g} + {\bf u}_0.
\]
Note that ${\bf N}$ and ${\bf T}$ are non-square, ${\bf \Gamma}_0$ is very sparse, and ${\bf T}$ is bidiagonal.
\end{lem}
\noindent In fact, assumption \eqref{infsup1} allows us to write ${\bf S}$ in an even simpler form.

\begin{lem} \label{lem:ourform}
The matrix ${\bf S}$ can be written as
  \begin{align*}
    {\bf S} = {\bf A}_t \kron ({\bf M}_{\vec x} {\bf K}_{\vec x} {\bf M}_{\vec x}) &+ {\bf M}_t \kron ({\bf A}_{\vec x} {\bf K}_{\vec x} {\bf A}_{\vec x}) + {\bf L}^\top \kron ({\bf M}_{\vec x} {\bf K}_{\vec x} {\bf A}_{\vec x}) \\
      &+ {\bf L} \kron ({\bf A}_{\vec x} {\bf K}_{\vec x} {\bf M}_{\vec x}) + {\bf \Gamma}_0 \kron {\bf M}_{\vec x}
    \end{align*}
  where
  \[
    {\bf L} := \langle \tfrac{\D}{\D t} \Phi_t^\delta, \Phi_t^\delta \rangle_{L_2(I)}, \quad
    {\bf M}_t := \langle \Phi_t^\delta, \Phi_t^\delta \rangle_{L_2(I)}, \quad
    {\bf A}_t := \langle \tfrac{\D}{\D t} \Phi_t^\delta, \tfrac{\D}{\D t} \Phi_t^\delta \rangle_{L_2(I)}.
    \]
    This matrix representation does not depend on $Y^\delta_t$ or $\Xi^\delta$ at all.
\end{lem}
\begin{proof}
The expansion of ${\bf B} := {\bf T} \kron {\bf M}_{\vec x} + {\bf N} \kron {\bf A}_{\vec x}$ in ${\bf S}$ yields a sum of five Kronecker products, one of which is
\begin{align*}
    ({\bf T}^\top \kron {\bf M}_{\vec x}) {\bf K}_Y ({\bf T} \kron {\bf A}_{\vec x})
     = ({\bf T}^\top {\bf O}^{-1} {\bf N}) \kron ({\bf M}_{\vec x} {\bf K}_{\vec x} {\bf A}_{\vec x}).
\end{align*}
We will show that ${\bf T}^\top {\bf O}^{-1} {\bf N} = {\bf L}^\top$; similar arguments hold for the other terms.
Thanks to $X^\delta_t \subset Y^\delta_t$, we can define the trivial embedding $F^\delta_t : X^\delta_t \to Y^\delta_t$. Defining
\begin{alignat*}{2}
T^\delta &\colon X_t^\delta \to {Y_t^\delta}',& \quad (T^\delta u)(v) &:= \langle \tfrac{\D}{\D t} u, v \rangle_{L_2(I)},\\
M^\delta &\colon Y_t^\delta \to {Y_t^\delta}',& \quad (M^\delta u)(v) &:= \langle u,v\rangle_{L_2(I)},
\end{alignat*}
we find ${\bf O} = \new{({\mathcal F_{\Xi^\delta}})}' M^\delta \mathcal F_{\Xi^\delta}$, ${\bf N} = \new{({\mathcal F_{\Xi^\delta}})}' M^\delta \mathcal F_t^\delta \mathcal  F_{\Phi_t^\delta}$ and ${\bf T} = \new{({\mathcal F_{\Xi^\delta}})'} T^\delta \mathcal F_{\Phi_t^\delta}$, \new{so}
\[
{\bf T}^\top {\bf O}^{-1} {\bf N} = \new{({\mathcal F_{\Phi_t^\delta}})}' {T^\delta}' F^\delta_t \mathcal F_{\Phi_t^\delta} = \langle \Phi_t, \tfrac{\D}{\D t} \Phi_t \rangle_{L_2(I)} = {\bf L}^\top.\qedhere
\]
\end{proof}

\subsection{Parallel complexity}
The \emph{parallel complexity} of our algorithm is the asymptotic runtime of
solving~\eqref{eqn:wit-mat} for ${\bf u} \in \R^{N_t N_{\vec x}}$ in terms of
$N_t := \dim X^\delta_t$ and $N_{\vec x} := \dim X^\delta_{\vec x}$,
given sufficiently many parallel processors and assuming no communication cost.

We understand the serial (resp.~parallel) cost of a matrix ${\bf B}$,
denoted $C^s_{\bf B}$ (resp.~$C^p_{\bf B}$), as the asymptotic runtime of performing
${\bf x} \mapsto {\bf B} {\bf x} \in \R^N$ in terms of $N$, on a single (resp.~sufficiently many) processors at no communication cost.
For \emph{uniformly sparse} matrices, i.e.~with $\mathcal O(1)$ nonzeros per row and column, the serial cost is $\mathcal O(N)$, and the parallel cost is $\mathcal O(1)$ by computing each cell of the output concurrently.

From Theorem~\ref{thm:DX-precond}, we see that ${\bf K}_X$ is such that $\kappa_2({\bf K}_X \hat {\bf S}) \lesssim 1$
uniformly in $\delta \in \Delta$. Therefore, for a given algebraic error tolerance $\eps$,
we require $\mathcal O(\log \eps^{-1})$ \new{PCG iterations.} Assuming that the parallel cost of matrices dominates that of vector addition and inner products,
the parallel complexity of a single PCG iteration is dominated by the cost of applying ${\bf K}_X$ and $\hat {\bf S}$.
As $\new{\hat {\bf S}} = {\bf W}^\top{\bf S}{\bf W}$, our algorithm runs in \new{complexity}
\new{\begin{equation} \label{eqn:cost}
\mathcal O(\log \eps^{-1}[C^\circ_{{\bf K}_X} + C^\circ_{{\bf W}^\top} +  C^\circ_{\bf S} + C^\circ_{{\bf W}}]) \quad (\circ \in \{s,p\}).
\end{equation}}

\begin{thm}\label{thm:main}
For fixed algebraic error tolerance $\eps > 0$, our algorithm runs in
\begin{itemize}
    \item serial complexity $\mathcal O(N_t N_{\vec x})$;
    \item time-parallel complexity $\mathcal O(\log (N_t) N_{\vec x})$;
    \item space-time-parallel complexity $\mathcal \new{O(\log (N_t N_{\vec x})})$.
\end{itemize}
\end{thm}

\begin{proof}
\new{We absorb the constant factor $\log \eps^{-1}$ of~\eqref{eqn:cost} into $\mathcal O$.}
We analyse the cost of every matrix separately.

\subsubsection{The (inverse) wavelet transform.}
As ${\bf W} = {\bf W}_t \kron {\mathrm {\bf Id}}_{\vec x}$, its serial cost equals $\mathcal O( C^s_{{\bf W}_t} N_{\vec x})$.
The choice of wavelet allows performing ${\bf x} \mapsto {\bf W}_t {\bf x}$ at linear serial cost (cf.~\S\ref{sec:wavelets}), so that $C^s_{\bf W} = \mathcal O(N_t N_{\vec x})$.

Using~\eqref{eqn:waveletcomposition}, we write ${\bf W}_t$ as the composition of $J$ matrices, each uniformly sparse and hence at parallel cost $\mathcal{O}(1)$.
\new{Because the mesh in time is quasi-uniform, we have $J \eqsim \log N_t$}. We find that $C^p_{{\bf W}_t} = \mathcal O(J) = \mathcal O(\log N_t)$, so that the time-parallel cost of ${\bf W}$ equals $\mathcal O(\log(N_t) N_{\vec x})$. By exploiting spatial parallelism \new{as well}, we find $C^p_{\bf W} =  \mathcal O(\log N_t)$.
\new{Analogous arguments hold for ${\bf W}_t^\top$ and ${\bf W}^\top$.}

\subsubsection{The preconditioner.}
Recall that ${\bf K}_X := \blockdiag[{\bf K}_{|\lambda|} \new{\mathbf A_{\vec x}} {\bf K}_{|\lambda|}]_{\lambda}$.
\new{Since the cost of ${\bf K}_j$ is independent of $j$,} we see that
\[
 C^s_{{\bf K}_X} = \mathcal O\big(N_t \cdot (2 C^s_{{\bf K}_j} + C^s_{{\bf A}_{\vec x}})\big) = \mathcal O(2 N_t C^s_{{\bf K}_j} + N_t N_{\vec x}).
\]
Implementing the ${\bf K}_j$ as typical multiplicative multigrid solvers with linear serial cost, we find $C^s_{{\bf K}_X} = \mathcal O(N_t N_{\vec x})$.

Through temporal parallelism, we can apply each block of ${\bf K}_X$ concurrently, resulting in a time-parallel cost of $\mathcal O(2 C^s_{{\bf K}_j} + C^s_{{\bf A}_{\vec x}}) = \mathcal O(N_{\vec x})$.

By parallelizing in space as well, we reduce the cost of the uniformly sparse ${\bf A}_{\vec x}$ to $\mathcal O(1)$.
\new{The parallel cost of multiplicative multigrid on quasi-uniform triangulations is $\mathcal O(\log N_{\vec x})$; cf.~\cite{McBryan1991}. It follows that $C^p_{{\bf K}_X} = \mathcal O(\log N_{\vec x})$.}

\subsubsection{The Schur matrix.} Using Lemma~\ref{lem:schurmat}, we write ${\bf S} = {\bf B}^\top {\bf K}_Y {\bf B} + {\bf \Gamma}_0 \kron {\bf M}_{\vec x}$ where ${\bf B} = {\bf T} \kron {\bf M}_{\vec x} + {\bf N} \kron {\bf A}_{\vec x}$, which immediately reveals that
\begin{align*}
    C^s_{\bf S} &= C^s_{{\bf B}^\top} + C^s_{{\bf K}_Y} + C^s_{\bf B} + C^s_{{\bf \Gamma}_0}\cdot C^s_{\bf M} = \mathcal O(N_t N_{\vec x} + C^s_{{\bf K}_Y}), \quad \text{and} \\
    C^p_{\bf S} &= \max \big\{ C^p_{{\bf B}^\top} + C^p_{{\bf K}_Y} + C^p_{\bf B}, ~~ C^p_{{\bf \Gamma}_0} \cdot C^p_{\bf M} \big\} = \mathcal O(C^p_{{\bf K}_Y})
\end{align*}
because every matrix except ${\bf K}_Y$ is uniformly sparse. With arguments similar to the previous paragraph, we see that ${\bf K}_Y$ (and hence ${\bf S}$) has serial cost $\mathcal O(N_t N_{\vec x})$, time-parallel cost $\mathcal O(N_{\vec x})$, and space-time-parallel cost $\mathcal O(\log N_{\vec x})$.
\end{proof}

\subsection{\new{Solving to higher accuracy}}
Instead of \emph{fixing} the algebraic error tolerance, maybe more realistic is is to desire a solution $\new{\tilde u^\delta} \in X^\delta$ for which the error \new{is proportional} to the discretization error, i.e.~$\|u - \tilde u^\delta\|_X \lesssim \inf_{u_\delta \in X^\delta} \|u - u_\delta\|_X$.

\new{Assuming that this error decays with a (problem-dependent) rate $s > 0$, i.e.~$\inf_{u_\delta \in X^\delta} \|u - u_\delta\|_X \lesssim (N_t N_{\vec x})^{-s}$, then the same holds for the solution $u^\delta$ of~\eqref{m9}; cf.~Thm.~\ref{thm:quasiopt}. When the algebraic error tolerance decays as $\eps \lesssim (N_t N_{\vec x})^{-s}$, a triangle inequality and~\eqref{eqn:stopping} show that the error of our solution $\tilde u^\delta$ obtained by PCG decays at rate $s$ too.}

\new{In this case, $\log \epsilon^{-1} = \mathcal O(\log(N_t N_{\vec x}))$. From~\eqref{eqn:cost} and the proof of Theorem~\ref{thm:main}, we find our algorithm to run in superlinear serial complexity $\mathcal O(N_t N_{\vec x} \log (N_t N_{\vec x}))$, time-parallel complexity $\mathcal O(\log^2(N_t) \log(N_{\vec x}) N_{\vec x})$, and polylogarithmic complexity $\mathcal O(\logt(N_t N_{\vec x}))$ parallel in space and time.}

\new{For elliptic PDEs, algorithms are available that offer quasi-optimal solutions, serially in linear complexity $\mathcal O(N_{\vec x})$---the cost of a serial solve to \emph{fixed} algebraic error---and in parallel in $\mathcal O(\logt N_{\vec x})$, by combining a \emph{nested iteration} with parallel multigrid; cf.~\cite[Ch.~5]{Hackbusch1985} and~\cite{Brandt1981}.}

In~\cite{Horton1995}, the question is posed whether ``good serial algorithms for parabolic PDEs
are intrinsically as parallel as good serial algorithms for elliptic PDEs'', basically
asking if the \new{lower bound of $\mathcal O(\logt (N_t N_{\vec x}))$} can be attained \new{by an algorithm that runs serially in $\mathcal O(N_t N_{\vec x})$; see~\cite[\S 2.2]{Worley1991} for a formal discussion}.

\new{Nested iteration drives down the serial complexity of our algorithm to a linear $\mathcal O(N_t N_{\vec x})$, and also improves the time-parallel complexity to $\mathcal O(\log (N_t) N_{\vec x})$.\footnote{\new{Interestingly, nested iteration offers no improvements parallel in space \emph{and} time, with complexity still $\mathcal O(\logt(N_t N_{\vec x}))$}.}}
\new{This is on par with the best-known results for elliptic problems, so}
we answer the question posed in~\cite{Horton1995} in the affirmative.

\section{Numerical experiments}\label{sec:numres}
We take the simple heat equation, \new{i.e.~$D = \mathrm{\bf Id}_{\vec x}$ and $c = 0$}. We select $p_t = p_{\vec x} = 1$, i.e.~lowest order finite elements in space and time.
We will use the 3-point wavelet introduced in~\cite{Stevenson1998}.

We implemented our algorithm in Python using the open source finite element library \texttt{NGSolve}~\cite{Schoberl2014} for meshing and discretization of the bilinear forms in space and time, \texttt{MPI} through \texttt{mpi4py} \cite{Dalcin2005} for distributed computations, and \texttt{SciPy}~\cite{Virtanen2020} for the sparse matrix-vector computations. \new{The source code is available at~\cite{VanVenetie2020b}.}

\subsection{Preconditioner calibration on a 2D problem}
Our wavelet-in-time, multigrid-in-space preconditioner is optimal: $\kappa_2({\bf K}_X \hat {\bf S}) \lesssim 1$. Here we will investigate
this condition number quantitatively.

As a model problem, we partition the temporal interval $I$ uniformly into $2^J$ subintervals. We consider the domain $\Omega := [0,1]^2$, and triangulate it uniformly into $4^K$ triangles. We set $N_t := \dim X^\delta_t = 2^J + 1$ and $N_{\vec x} := \dim X^\delta_{\vec x} = (2^K - 1)^2$.

We start by using \new{direct inverses ${\bf K}_j = ({\bf A}_{\vec x} + 2^j {\bf M}_{\vec x})^{-1}$} and \new{${\bf K}_{\vec x} = {\bf A}_{\vec x}^{-1}$} to
determine the best possible condition numbers.
We found that replacing ${\bf K}_j$ by \new{${\bf K}_j^{\alpha} = (\alpha {\bf A}_{\vec x} + 2^j {\bf M}_{\vec x})^{-1}$} for $\alpha = 0.3$ gave better conditioning; see also the left of Table~\ref{tbl:direct-inverse}.
At the right of \new{Table~\ref{tbl:direct-inverse}}, we see that the condition numbers are very robust with respect to spatial refinements, but less so for refinements in time. Still, at \new{$N_t = 16\,129$}, we observe a modest $\kappa_2({\bf K}_X \hat {\bf S})$ of $8.74$.

Replacing the direct inverses with multigrid solvers, we found a good balance between speed and conditioning at 2 V-cycles with 3 Gauss-Seidel smoothing steps per grid. We decided to use these for our experiments.

\begin{table}
\vspace{-1em}
\centering
\begin{minipage}[t]{0.28\linewidth}
\strut\vspace*{-\baselineskip}\newline
\includegraphics[trim=10 30 10 46.5, clip, width=\linewidth]{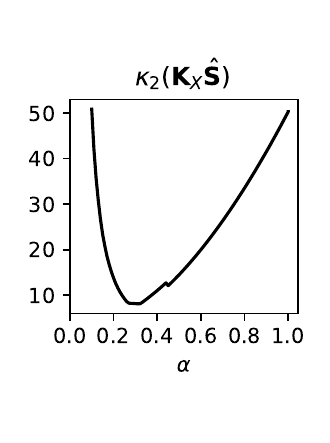}
\end{minipage}
\begin{minipage}[t]{0.7\linewidth}
\strut\vspace*{-\baselineskip}\newline
\resizebox{\linewidth}{!}{%
\begin{tabular}{rrrrrrrrrrr}
\toprule
\multicolumn{2}{r}{$N_t = 65$}   &  129  &  257  &  513  &        1\,025 & 2\,049 & 4\,097 & 8\,193 \\
\midrule
$N_{\vec x} = 49$    &  6.34 &  7.05 &        7.53 &        7.89 &        8.15 & 8.37 & 8.60 & 8.78 \\
225   &  6.33 &  6.89 &        7.55 &        7.91 &        8.14 & 8.38 & 8.57 & 8.73 \\
961   &  6.14 &  6.89 &        7.55 &        7.93 &        {\bf 8.15}  & 8.38 & 8.57 & 8.74\\
3\,969  &  6.14 &  7.07 &        7.56 &        7.87 &        8.16 & 8.38 & 8.57 & 8.74 \\
16\,129 &  6.14 &  6.52 &  7.55 & 7.86 & 8.16 & 8.37 & 8.57 & 8.74 \\
\bottomrule
\end{tabular}}
\end{minipage}
\label{tbl:direct-inverse}
\caption{Computed condition numbers $\kappa_2({\bf K}_X \hat {\bf S})$. Left: fixed $N_t = 1025$, $N_{\vec x} = 961$ for varying $\alpha$. Right: fixed $\alpha = 0.3$ for varying $N_t$ and $N_{\vec x}$.}
\vspace{-3em}
\end{table}

\subsection{Time-parallel results}
We perform computations on Cartesius, the Dutch supercomputer. Each Cartesius node has 64GB of memory and 12 cores (at 2 threads per core) running at 2.6GHz.
Using the preconditioner detailed above, we iterate PCG on~\eqref{eqn:wit-mat}
with ${\bf S}$ computed as in Lemma~\ref{lem:ourform},
until achieving an algebraic error of $\eps = 10^{-6}$; see also \S\ref{sub:solving}.
For the spatial multigrid solvers, we use $2$ V-cycles with $3$ Gauss-Seidel smoothing steps per grid.

\subsubsection{Memory-efficient time-parallel implementation.}
For ${\bf X} \in \R^{N_{\vec x} \times N_t}$, we define $\Vec({\bf X}) \in \R^{N_t N_{\vec x}}$ as
the vector obtained by stacking columns of ${\bf X}$ vertically.
For memory efficency, we do not build matrices of the form ${\bf B}_t \kron {\bf B}_{\vec x}$ appearing in Lemma~\ref{lem:ourform}
directly, but instead perform matrix-vector products using the identity
\begin{equation} \label{eqn:kron-id}
  ({\bf B}_t \kron {\bf B}_{\vec x}) \Vec({\bf X}) = \Vec({\bf B}_{\vec x} ({\bf B}_t {\bf X}^\top)^\top) =
(\mathrm{\bf Id}_t \kron {\bf B}_{\vec x})\Vec({\bf B}_t {\bf X}^\top).
\end{equation}

Each parallel processor stores only a subset of the temporal degrees of freedom, e.g.~a subset of columns of ${\bf X}$. When ${\bf B}_t$ is uniformly sparse, which holds true for all of our temporal matrices,
using~\eqref{eqn:kron-id} we can evaluate $({\bf B}_t \kron {\bf B}_{\vec x}) \Vec({\bf X})$
in $\mathcal O(C^s_{{\bf B}_{\vec x}})$ operations parallel in time\new{: on}
each parallel processor, we compute `our' columns of ${\bf Y} := {\bf B}_t {\bf X}^\top$ by receiving the necessary columns of ${\bf X}$ from neighbouring processors, and then compute ${\bf B}_{\vec x} {\bf Y}^\top$ without communication.

The preconditioner ${\bf K}_X$ is block-diagonal, \new{making its time-parallel application trivial}. Representing the wavelet transform of \S\ref{sec:wavelets} as the composition of $J$ Kronecker products allows a time-parallel implementation using the above.

\subsubsection{2D problem.} We select $\Omega := [0,1]^2$ with a uniform triangulation $\mathcal T_{\Omega}$, and we triangulate $I$ uniformly into $\mathcal T_I$. We select the smooth solution $u(t,x,y) := \exp(-2 \pi^2 t) \sin(\pi x) \sin(\pi y)$, \new{so} the problem has vanishing forcing data \new{$g$}.

Table~\ref{tbl:2d-strong} details the strong scaling results, i.e.~fixing the problem size and increasing the number of processors $P$.
We triangulate $I$ into $2^{14}$ time slabs, yielding $N_t = 16\,385$ temporal degrees of freedom, and $\Omega$ into $4^{8}$ triangles, yielding a $X^\delta_{\vec x}$ of dimension $N_{\vec x} = 65\,025$.
The resulting system contains $1\,065\,434\,625$ degrees of freedom and our solver reaches the algebraic error tolerance after 16 iterations. In perfect strong scaling, the total number of CPU-hours remains constant.
Even at 2\,048 processors, we observe a parallel efficiency of around $92.9\%$, solving this system in a modest 11.7 CPU-hours.
Acquiring strong scaling results on a single node was not possible due to memory limitations.
\begin{table}[p]
\centering
\scalebox{0.9}{%
\begin{tabular}{rrrrrrrr}
\toprule
 \thead{$P$} &     \thead{$N_t$} &        \thead{$N_{\vec x}$} & \thead{$N = N_t N_{\vec x}$} &  \thead{its} &   \thead{time (s)}  &\thead{time/it (s)}&  \thead{CPU-hrs} \\
\midrule
   1--16 &  16\,385 &  65\,025 &  1\,065\,434\,625 & \multicolumn{4}{c}{\Vhrulefill~out of memory~\Vhrulefill} \\
      32 &  16\,385 &  65\,025 &  1\,065\,434\,625 &     16 & 1224.85 &  76.55 &  10.89 \\
      64 &  16\,385 &  65\,025 &  1\,065\,434\,625 &     16 &  615.73 &  38.48 &  10.95 \\
     128 &  16\,385 &  65\,025 &  1\,065\,434\,625 &     16 &  309.81 &  19.36 &  11.02 \\
     256 &  16\,385 &  65\,025 &  1\,065\,434\,625 &     16 &  163.20 &  10.20 &  11.61 \\
     512 &  16\,385 &  65\,025 &  1\,065\,434\,625 &     16 &   96.54 &   6.03 &  13.73 \\
     512 &  16\,385 &  65\,025 &  1\,065\,434\,625 &     16 &   96.50 &   6.03 &  13.72 \\
  1\,024 &  16\,385 &  65\,025 &  1\,065\,434\,625 &     16 &   45.27 &   2.83 &  12.88 \\
  2\,048 &  16\,385 &  65\,025 &  1\,065\,434\,625 &     16 &       20.59 &                 1.29 &           11.72 \\
\bottomrule
\end{tabular}}
\vspace{0.5em}
\caption{Strong scaling results for the 2D problem.}
\label{tbl:2d-strong}
\end{table}
\begin{table}[p]
\centering
\scalebox{0.9}{%
\begin{tabular}{rrrrrrrrrr}
\toprule
 &\thead{$P$} &     \thead{$N_t$} &        \thead{$N_{\vec x}$} & \thead{$N = N_t N_{\vec x}$} &  \thead{its} &   \thead{time (s)} &  \thead{time/it (s)}&  \thead{CPU-hrs} \\
\midrule
\parbox[c]{1em}{\multirow{5}{*}{\rotatebox[origin=c]{90}{{\!\!single node}}}}
&       1 &        9 &  261\,121 &       2\,350\,089 &      8 &       33.36 &  4.17 &   0.01 \\
&       2 &       17 &  261\,121 &       4\,439\,057 &     11 &       46.66 &  4.24 &   0.03 \\
&       4 &       33 &  261\,121 &       8\,616\,993 &     12 &       54.60 &  4.55 &   0.06 \\
&       8 &       65 &  261\,121 &      16\,972\,865 &     13 &       65.52 &  5.04 &   0.15 \\
&      16 &      129 &  261\,121 &      33\,684\,609 &     13 &       86.94 &  6.69 &   0.39 \\\midrule \parbox[c]{1em}{\multirow{7}{*}{\rotatebox[origin=c]{90}{{\!\!multiple nodes}}}}
&      32 &      257 &  261\,121 &      67\,108\,097 &     14 &       93.56 &  6.68 &   0.83 \\
&      64 &      513 &  261\,121 &     133\,955\,073 &     14 &       94.45 &  6.75 &   1.68 \\
&     128 &     1\,025 &  261\,121 &     267\,649\,025 &     14 &       93.85 &  6.70 &   3.34 \\
&     256 &     2\,049 &  261\,121 &     535\,036\,929 &     15 &      101.81 &  6.79 &   7.24 \\
&     512 &     4\,097 &  261\,121 &  1\,069\,812\,737 &     15 &      101.71 &  6.78 &  14.47 \\
&  1\,024 &     8\,193 &  261\,121 &  2\,139\,364\,353 &     16 &      108.32 &  6.77 &  30.81 \\
&  2\,048 &  16\,385 &  261\,121 &  4\,278\,467\,585 &     16 &      109.59 &  6.85 &  62.34 \\
\bottomrule
\end{tabular}}
\vspace{0.5em}
\caption{Weak scaling results for the 2D problem.}
\label{tbl:2d-weak}
\end{table}
\begin{table}[p]
\centering
\scalebox{0.9}{%
\begin{tabular}{rrrrrrrr}
\toprule
 \thead{$P$} &     \thead{$N_t$} &        \thead{$N_{\vec x}$} & \thead{$N = N_t N_{\vec x}$} &  \thead{its} &   \thead{time (s)}  &\thead{time/it (s)}&  \thead{CPU-hrs} \\
\midrule
   1--64 &  16\,385 &  250\,047 &  4\,097\,020\,095 & \multicolumn{4}{c}{\Vhrulefill~out of memory~\Vhrulefill} \\
     128 &  16\,385 &  250\,047 &  4\,097\,020\,095 &     18 &     3\,308.49 &  174.13 &  117.64 \\
     256 &  16\,385 &  250\,047 &  4\,097\,020\,095 &     18 &     1\,655.92 &   87.15 &  117.75 \\
     512 &  16\,385 &  250\,047 &  4\,097\,020\,095 &     18 &        895.01 &   47.11 &  127.29 \\
  1\,024 &  16\,385 &  250\,047 &  4\,097\,020\,095 &     18 &        451.59 &   23.77 &  128.45 \\
  2\,048 &  16\,385 &  250\,047 &  4\,097\,020\,095 &     18 &        221.12 &   12.28 &  125.80 \\
\bottomrule
\end{tabular}}
\vspace{0.5em}
\caption{Strong scaling results for the 3D problem.}
\label{tbl:3d-strong}
\end{table}
\begin{table}[p]
\centering
\scalebox{0.9}{%
\begin{tabular}{rrrrrrrr}
\toprule
 \thead{$P$} &     \thead{$N_t$} &        \thead{$N_{\vec x}$} & \thead{$N = N_t N_{\vec x}$} &  \thead{its} &   \thead{time (s)}  &\thead{time/it (s)}&  \thead{CPU-hrs} \\
\midrule
    16 &     129 & 250\,047 &     32\,256\,063 & 15 & 183.65 & 12.24 &   0.82 \\
    32 &     257 & 250\,047 &     64\,262\,079 & 16 & 196.26 & 12.27 &   1.74 \\
    64 &     513 & 250\,047 &    128\,274\,111 & 16 & 197.55 & 12.35 &   3.51 \\
   128 &  1\,025 & 250\,047 &    256\,298\,175 & 17 & 210.21 & 12.37 &   7.47 \\
   256 &  2\,049 & 250\,047 &    512\,346\,303 & 17 & 209.56 & 12.33 &  14.90 \\
   512 &  4\,097 & 250\,047 & 1\,024\,442\,559 & 17 & 210.14 & 12.36 &  29.89 \\
1\,024 &  8\,193 & 250\,047 & 2\,048\,635\,071 & 18 & 221.77 & 12.32 &  63.08 \\
2\,048 & 16\,385 & 250\,047 & 4\,097\,020\,095 & 18 & 221.12 & 12.28 & 125.80 \\
\bottomrule
\end{tabular}}
\vspace{0.5em}
\caption{Weak scaling results for the 3D problem.}
\label{tbl:3d-weak}
\end{table}

Table~\ref{tbl:2d-weak} details the weak scaling results, i.e.~fixing the problem size per processor and increasing the number of processors. In perfect weak scaling, the time per iteration should remain constant.
We observe a slight increase in time per iteration on a single node, but when scaling to multiple nodes, we observe a near-perfect parallel efficiency of around 96.7\%, solving the final system with $4\,278\,467\,585$ degrees of freedom  in a mere 109 seconds.

\subsubsection{3D problem.} We select $\Omega := [0,1]^3$, and prescribe  the solution $u(t,x,y,z) := \exp(-3 \pi^2 t) \sin(\pi x) \sin(\pi y) \sin(\pi z)$, so the problem has vanishing forcing data \new{$g$}.

Table~\ref{tbl:3d-strong} shows the strong scaling results. We triangulate $I$ uniformly into $2^{14}$ time slabs, and $\Omega$ uniformly into $8^6$ tetrahedra. The arising system has $N = 4\,097\,020\,095$ unknowns, which we solve to tolerance in 18 iterations. The results are comparable to those in two dimensions, albeit a factor two slower at similar problem sizes.

Table~\ref{tbl:3d-weak} shows the weak scaling results for the 3D problem. As in the two-dimensional case, we observe excellent scaling properties, and see that the time per iteration is nearly constant.

\section{Conclusion}\label{sec:conclusion}
We have presented a framework for solving linear parabolic evolution equations  massively in parallel.
Based on earlier ideas \cite{Andreev2016,Neumuller2019,Stevenson2020}, we found a remarkably simple
symmetric Schur-complement equation. With a tensor-product discretization of the space-time cylinder
using standard finite elements in time and space together with a wavelet-in-time multigrid-in-space
preconditioner, \new{we were able to solve the arising systems to fixed accuracy} in a uniformly bounded number of \new{PCG} steps.

We found that our algorithm runs in linear complexity on a single
processor. Moreover, when \emph{sufficiently many} parallel processors are
available and communication is free, its runtime scales \emph{logarithmically}
in the discretization size.
These complexity results translate to a highly efficient algorithm in practice.

The numerical experiments serve as a showcase for the described space-time method, and exhibit its excellent time-parallelism by solving a linear system with over 4 billion unknowns in just 109 seconds, using just over 2 thousand parallel processors. By incorporating spatial parallelism as well, we expect these results to scale well to much larger problems.

Although performed in the rather
restrictive setting of the heat equation discretized using piecewise linear polynomials on
uniform triangulations, the parallel framework already allows solving more general linear parabolic
PDEs using polynomials of varying degree on locally refined (tensor-product) meshes.
In this more general setting, we envision load balancing to become the main hurdle in achieving good scaling results.

\subsubsection*{Acknowledgement.}
The authors would like to thank their advisor Rob Stevenson for the many fruitful discussions.

\subsubsection*{Funding.}
Both authors were supported by Netherlands Organization for Scientific Research (NWO) under contract no.~613.001.652. Computations were performed at the national supercomputer Cartesius under SURF code EINF-459.

%
%
\newcommand{\etalchar}[1]{$^{#1}$}

\end{document}